\documentclass[pdflatex,sn-mathphys-num]{sn-jnl}

\usepackage{bm,comment,multirow,mathtools} 
\usepackage{calc,enumitem,subcaption,tabularx} 

\usepackage{graphicx}%
\usepackage{amsmath,amssymb,amsfonts}%
\usepackage{amsthm}%
\usepackage{mathrsfs}%
\usepackage[title]{appendix}%
\usepackage{xcolor}%
\usepackage{textcomp}%
\usepackage{manyfoot}%
\usepackage{booktabs}%
\usepackage{algorithm}%
\usepackage{algorithmicx}%
\usepackage{algpseudocode}%
\usepackage{listings}%

\theoremstyle{thmstyleone}%
\newtheorem{theorem}{Theorem}
\newtheorem{proposition}{Proposition}
\newtheorem{lemma}{Lemma} 

\theoremstyle{thmstyletwo}%

\theoremstyle{thmstylethree}%
\newtheorem{definition}{Definition}%

\makeatletter
\def\orcid#1{\href{https://orcid.org/#1}{\includegraphics[height=1em]{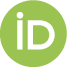}}}
\makeatother

\begin{document}

\title[Distributionally robust optimization for recommendation selection]
{Distributionally robust optimization for recommendation selection}


\author[1]{\fnm{Tomoya} \sur{Yanagi}}
\author[2]{\fnm{Shunnosuke} \sur{Ikeda} \orcid{0009-0004-4283-0819}}
\author[3]{\fnm{Ken} \sur{Kobayashi} \orcid{0000-0002-6609-7488}}
\author*[2,4]{\fnm{Yuichi} \sur{Takano} \orcid{0000-0002-8919-1282}}

\affil[1]{\orgdiv{Graduate School of Science and Technology}, \orgname{University of Tsukuba}, \orgaddress{\street{1--1--1 Tennodai}, \city{Tsukuba-shi}, \postcode{305--8573}, \state{Ibaraki}, \country{Japan}}}

\affil[2]{\orgdiv{Institute of Systems and Information Engineering}, \orgname{University of Tsukuba}, \orgaddress{\street{1--1--1 Tennodai}, \city{Tsukuba-shi}, \postcode{305--8573}, \state{Ibaraki}, \country{Japan}}}

\affil[3]{\orgdiv{Department of Industrial Engineering and Economics}, \orgname{Institute of Science Tokyo}, \orgaddress{\street{2--12--1 Ookayama}, \city{Meguro-ku}, \postcode{152--8552}, \state{Tokyo}, \country{Japan}}}

\affil[4]{\orgdiv{Center for Artificial Intelligence Research, Tsukuba Institute for Advanced Research (TIAR)}, \orgname{University of Tsukuba}, \orgaddress{\street{1--1--1 Tennodai}, \city{Tsukuba-shi}, \postcode{305--8577}, \state{Ibaraki}, \country{Japan}}}


\abstract{
Recommender systems play an essential role in online services by providing personalized item lists to support users' decision-making processes. 
While collaborative filtering methods can achieve high accuracy, it is crucial to consider not only accuracy but also the diversity of recommended items to improve user satisfaction. 
Although financial portfolio theory has been applied to balance these factors, existing models are often sensitive to estimation errors in rating statistics.
To overcome these challenges, we establish a computational framework of distributionally robust optimization (DRO) for recommendation selection. 
We first formulate a cardinality-constrained DRO model based on moment-based ambiguity sets to select a specified number of items for each user. 
We then design a penalty alternating direction method (PADM) to efficiently compute high-quality solutions and prove its convergence properties.
Computational experiments using three publicly available rating datasets demonstrate that our DRO model generates more diverse recommendations than existing models while maintaining the same level of accuracy. 
Additionally, our solution method computes these recommendations for each user in just a few seconds, proving its practical effectiveness. 
This study establishes a DRO framework that has the potential to enhance the recommendation quality of various collaborative filtering methods.}




\maketitle

\section{Introduction}\label{sec:intro}

\subsection{Background}

The main purpose of recommender systems is to support users' decision-making processes in online services by providing a personalized list of items that match individual preferences based on behavioral history and attribute information~\cite{Ag16}.
In today's information-overload environment fueled by advances in information and communication technologies, recommender systems play an essential role for both users and providers of online services.
Benefits to users include efficient information search, improved service satisfaction, discovery of new items, and expansion of interests.
Meanwhile, benefits to providers include increased service usage, increased sales and revenue, and a foundation for building an advantage over competitors.
In fact, recommender systems have been effectively implemented in a variety of online services~\cite{LuWu15}, such as Amazon~\cite{LiSm03}, Netflix~\cite{BeLa07}, Google News~\cite{DaDa07}, and YouTube~\cite{DaLi10}.

Although various algorithms have been developed for recommender systems~\cite{Ag16}, one of the most successful techniques is collaborative filtering, which can be categorized into neighborhood-based and model-based methods. 
Neighborhood-based methods exploit the similarity between users to infer the preferences of target users~\citep{HeKo99}, or the similarity between items to suggest related items~\citep{SaKa01}. 
Model-based methods use advanced machine learning algorithms (e.g., matrix factorization~\citep{KoBe09} and deep learning~\citep{ZhYa19}) to predict users' ratings for items with high accuracy.

In the practical use of recommender systems, recommendations are often presented to users in the form of a list of items (e.g., top-$N$ recommendations) rather than a single item~\citep{Ad13}.
It has been demonstrated that the most accurate recommendations are not necessarily the most helpful to users~\citep{McRi06}, so it is crucial to consider not only the accuracy of recommendations but also the diversity of recommended items and their novelty to users~\citep{VaCa11}.
In fact, user satisfaction with recommendation lists has been shown to be positively correlated with the diversity of those recommendations~\citep{EkHa14}. 
In addition, some well-known recommender algorithms tend to recommend already popular items, resulting in reduced sales diversity~\citep{FlHo09}. 

\subsection{Related Work}

A number of prior studies have considered the definition and evaluation of diversity, its impact on recommendation quality, and the development of diversification algorithms~\citep{CaHu22,KaBr16,KuPo17}.
\citet{ZiMc05} dealt with decreasing the intra-list similarity of topics in recommendation lists. 
\citet{LaHa10} investigated temporal diversity of recommendations to avoid repeatedly recommending the same items to a particular user. 
\citet{AdKw12} considered a general-purpose item ranking algorithm for diverse recommendations based on the popularity of each item.
\citet{KaIw13} developed a probabilistic diversification model for increasing the probability of recommending at least one relevant item to each user. 
\citet{VaBa14} provided a binomial framework for defining the genre diversity of recommendations. 

We here focus on the application of financial portfolio theory~\citep{ElGr09} as a method for achieving both accuracy and diversity in recommendations. 
This approach was first introduced in information retrieval~\citep{WaZh09} and then applied to collaborative filtering~\citep{Wa09}. 
\citet{Wa09} demonstrated that a ranking strategy based on mean--variance portfolio analysis can improve the recommendation accuracy of collaborative filtering algorithms. 
\citet{ShZh12} devised a mean--variance ranking strategy using a latent factor model of user ratings.
\citet{HuZh11} developed heuristic algorithms for several optimization models with cardinality constraints to select diverse lists of items.  
\citet{XiPe20} used the portfolio optimization model to recommend relevant services based on service risk facets. 

These previous studies have demonstrated that applying mean-variance portfolio optimization to recommender systems can improve the quality (e.g., accuracy, diversity, and novelty) of recommendations.
However, estimation of the mean and covariance of ratings is required for mean–variance portfolio optimization, and these statistics are subject to inevitable estimation errors.
It has been pointed out that such estimation errors degrade the performance of mean--variance portfolio optimization models~\cite{Br93,Mi89}.
\citet{YaYa26} developed shrinkage estimation methods for accurately estimating the rating covariance matrix. 
Furthermore, to cope with uncertainty in both the mean and covariance of ratings, \citet{YaIk24} proposed a robust portfolio optimization model based on the cardinality-based uncertainty sets~\cite{BeSi04}. 
However, this robust optimization model has theoretical drawbacks, such as not guaranteeing the positive semidefiniteness of covariance matrices and ignoring the relationship between the uncertainty of the mean and covariance.

To overcome the challenges faced by these existing methods, we exploit the framework of distributionally robust optimization (DRO)~\cite{KuSh25,RaMe22}, where the optimal solution is derived in terms of the worst-case probability distribution contained in the ambiguity set.
This framework combines the advantages of robust optimization, which focuses on worst-case scenarios, and stochastic optimization, which considers probability distributions.
Recently, \citet{KoTa23} developed an algorithm for solving the cardinality-constrained distributionally robust portfolio optimization model based on moment-based ambiguity sets~\cite{DeYe10}.
To the best of our knowledge, however, no prior studies have considered a DRO model for selecting a high-quality list of recommended items.

\subsection{Contribution}

The main goal of this paper is to establish a computational DRO framework for recommendation selection. 
Specifically, we first formulate a DRO model with a cardinality constraint to select a specified number of recommendations for each user.
Next, we design a penalty alternating direction method (PADM)~\cite{geissler2017penalty,moreira2022alternating} for efficiently computing high-quality solutions to our DRO model. 
Furthermore, we prove the convergence properties of our PADM algorithm based on two revised problem formulations. 

To evaluate the effectiveness of our recommendation framework, we conducted computational experiments using three publicly available rating datasets. 
The experimental results demonstrate that our model generates more diverse recommendations than existing recommendation selection models while maintaining the same level of recommendation accuracy.
We also experimentally verify that our solution method computes the aforementioned high-quality recommendations for each user in just a few seconds.

\subsection{Notation}
The zero and all-ones vectors of appropriate sizes are denoted by $\bm{0}$ and $\bm{1}$, respectively. 
The zero and identity matrices of appropriate sizes are denoted by $\bm{O}$ and $\bm{I}$, respectively. 
The $L_p$-norm of a vector is denoted by $\|\cdot\|_p$, and the Frobenius norm of a matrix is denoted by $\|\cdot\|_{\rm F}$. 
The trace (i.e., sum of the diagonal entries) of a matrix is denoted by $\mathrm{tr}(\,\cdot\,)$. 
For two symmetric matrices $\bm{X}$ and $\bm{Y}$, we write $\bm{X} \succeq \bm{Y}$ if $\bm{X} - \bm{Y}$ is a positive semidefinite matrix. 
The standard inner product of matrices $\bm{A} \coloneqq (a_{ij})_{(i,j) \in I \times J}$ and $\bm{B} \coloneqq (b_{ij})_{(i,j) \in I \times J}$ is defined as $\bm{A} \bullet \bm{B} \coloneqq \sum_{(i,j) \in I \times J} a_{ij}b_{ij}$.

\section{Problem Formulation}\label{sec:dis_rob}

In this section, we first pose a nominal optimization model for recommendation selection, and next introduce moment-based ambiguity sets~\cite{DeYe10} for the probability distribution of ratings. 
We then present, by following \citet{KoTa23}, our DRO model with a cardinality constraint to select a specified number of recommendations for each user. 

\subsection{Nominal Optimization Model}

Let $r_{ui}$ be the rating value that user $u \in U$ gives to item $i \in I$ (e.g., through a 5-star rating system), where $U$ and $I$ are the sets of users and items, respectively. 
We define $R_{\rm O}$ and $R_{\rm M}$ as the sets of user-item pairs with observed and missing ratings, respectively: 
\[
\mbox{$r_{ui}$~is~}
\begin{cases}
~\mbox{observed} & \mbox{if~$(u,i) \in R_{\rm O}$}, \\
~\mbox{missing} & \mbox{if~$(u,i) \in R_{\rm M}$} \\
\end{cases} \quad ((u,i) \in U \times I).
\]
A key task in recommender systems is to predict the missing rating values for user--item pairs $(u,i) \in R_{\rm M}$. 
For this purpose, various rating prediction algorithms, such as collaborative filtering, have been proposed~\cite{Ag16}.

For each user $u \in U$, let $I(u) \subseteq \{i \in I \mid (u,i) \in R_{\rm M} \}$ be the set of candidate items to be recommended. 
For each user $u \in U$, we introduce $\bm{x} \coloneqq (x_{i})_{i \in I(u)} \in \{0,1\}^{\lvert I(u) \rvert}$ as a vector of binary decision variables for selecting recommendations; that is, $x_{i} = 1$ if item $i \in I(u)$ is recommended, and $x_{i} = 0$ otherwise. 
The problem of selecting $N$ recommendations for each user $u \in U$ can be posed as the following nominal optimization model: 
\begin{subequations}\label{eq:robust_portfolio}
\begin{align} 
    \underset{\bm{x}}{\text{maximize}} \quad &  \sum_{i \in I(u)} \hat{r}_{ui} x_{i} \label{eq:robust_portfolio_obj} \\ 
    \text {subject to} \quad & \sum_{i \in I(u)} x_{i} = N, \label{eq:robust_portfolio_N} \\
    & \bm{x} \in \{0,1\}^{|I(u)|}, \label{eq:robust_portfolio_x}
\end{align}
\end{subequations}
where $\hat{r}_{ui}$ is the predicted rating value for user--item pair $(u,i) \in R_{\rm M}$, and $N \in \mathbb{Z}_{+}~(N \le |I(u)|)$ is a parameter representing the number of recommended items. 
The objective in Eq.~\eqref{eq:robust_portfolio_obj} is to maximize the sum of the predicted ratings for the  recommended items. 
Eq.~\eqref{eq:robust_portfolio_N} is the cardinality constraint for specifying the number of recommended items. 
Eq.~\eqref{eq:robust_portfolio_x} is the binary constraint for selecting recommendations. 

Note that problem~\eqref{eq:robust_portfolio} is exactly consistent with the common top-$N$ recommendation~\citep{CrKo10}, which selects the $N$ items with the highest predicted ratings. 

\subsection{Moment-based Ambiguity Set}

We will now treat missing ratings as random variables by following financial portfolio theory~\citep{ElGr09}, which treats stock returns as random variables. 
Let $(\Omega, \mathcal{F})$ be the measurable space associated with the probability measures of missing ratings.
For each user $u \in U$, we define $\tilde{\bm{r}}_{u} \coloneqq (\tilde{r}_{ui})_{i \in I(u)} : \Omega \to \mathbb{R}^{|I(u)|}$ as an $\mathcal{F}$-measurable function (i.e., $|I(u)|$-dimensional random vector), consisting of the missing ratings for each item $i \in I(u)$. 

We define $\hat{\bm{\mu}}_{u} \coloneqq (\hat{\mu}_{ui})_{i \in I(u)} \in \mathbb{R}^{|I(u)|}$ as a vector of expected ratings, which can be specified using rating prediction algorithms (i.e., $\hat{\mu}_{ui} \coloneqq \hat{r}_{ui}$ for $(u,i) \in R_{\rm M}$). 
In contrast, it is virtually impossible to estimate the covariance of ratings between items for each user $u \in U$, mainly because the number of items rated by each user (i.e., $|I(u)|$) is usually very small. 
It is therefore reasonable to instead estimate a sample covariance matrix from all users' ratings. 

Let $U(i,j)$ be the set of users who rated both items $i,j \in I$.
A sample estimate of the covariance matrix for each user $u \in U$ is then given by 
\begin{equation}\label{eq:smpmat}
\hat{\bm{\Sigma}}_u \coloneqq (\hat{\sigma}_{ij})_{(i,j) \in I(u) \times I(u)} \in \mathbb{R}^{\lvert I(u) \rvert \times \lvert I(u) \rvert} \quad (u \in U),
\end{equation}
where 
\begin{align}
& \hat{\sigma}_{ij} := 
\begin{cases}
~ \displaystyle \frac{1}{\lvert U(i,j) \rvert} \sum_{u \in U(i,j)} \left(r_{ui} - \mu_i^{(j)}\right) \left(r_{uj} - \mu_j^{(i)}\right) & \mbox{if~$U(i,j) \neq \emptyset$,} \\
~0 & \mathrm{otherwise}
\end{cases} \quad ((i,j) \in I \times I), \notag \\ 
& \mu_i^{(j)} := \frac{1}{\lvert U(i,j) \rvert} \sum_{u \in U(i,j)} r_{ui} \quad ((i,j) \in I \times I). \notag 
\end{align}

Let $\mathcal{M}$ be the set of all probability measures in the measurable space $(\Omega, \mathcal{F})$, and $\mathbb{E}_P[\,\cdot\,]$ be the expectation under the probability measure $P \in \mathcal{M}$.
If the support of probability distribution of $\tilde{\bm{r}}_{u}$ is $\mathbb{R}^{|I(u)|}$, the associated moment-based ambiguity set~\cite{DeYe10} is given by
\begin{align}
& \mathcal{A}_{u}(\hat{\bm{\mu}}_u, \hat{\bm{\Sigma}}_u, \kappa_1, \kappa_2) \notag \\
\coloneqq~ & \left\{ P \in \mathcal{M} \,\middle|\, 
\begin{array}{l}
( \mathbb{E}_P[\tilde{\bm{r}}_u] - \hat{\bm{\mu}}_u )^\top \hat{\bm{\Sigma}}_u^{-1} ( \mathbb{E}_P[\tilde{\bm{r}}_u] -  \hat{\bm{\mu}}_u ) \leq \kappa_1 \\
\mathbb{E}_P [ (\tilde{\bm{r}}_u - \hat{\bm{\mu}}_u)(\tilde{\bm{r}}_u - \hat{\bm{\mu}}_u)^\top ] \preceq \kappa_2 \hat{\bm{\Sigma}}_u 
\end{array}
\right\} \quad (u \in U), \label{eq:set_unc}
\end{align}
where $\kappa_1 \in \mathbb{R}_{+}$ and $\kappa_2 \in \mathbb{R}_{+}$ are parameters that represent the ambiguity of $\hat{\bm \mu}_u$ and $\hat{\bm \Sigma}_u$, respectively.  
The first condition in Eq.~\eqref{eq:set_unc} ensures that the expectation of $\tilde{\bm{r}}_u$ lies inside an ellipsoid of size $\kappa_1$ centered at $\hat{\bm \mu}_u$.
The second condition in Eq.~\eqref{eq:set_unc} states that the covariance matrix of $\tilde{\bm{r}}_u$ is upper bounded by $\kappa_2 \hat{\bm \Sigma}_u$ via the matrix inequality. 

\subsection{Distributionally Robust Optimization Model}

We are now in a position to formulate our DRO model for selecting $N$ recommendations for each user $u \in U$ as follows:
\begin{subequations}\label{eq:dist_robust_portfolio}
\begin{align}
    \underset{\bm{x}}{\text{maximize}} \quad & \min \left\{ \mathbb{E}_{P} \Bigl[ \sum_{i \in I(u)} \tilde{r}_{ui} x_{i} \Bigr] \,\middle|\, P \in \mathcal{A}_{u}(\hat{\bm{\mu}}_u, \hat{\bm{\Sigma}}_u, \kappa_1, \kappa_2) \right\} \label{eq:dist_robust_portfolio_obj} \\
    \text {subject to} \quad & \sum_{i \in I(u)} x_{i} = N, \label{eq:dist_robust_portfolio_N} \\
    & \bm{x} \in \{0,1\}^{|I(u)|}. \label{eq:dist_robust_portfolio_x}
\end{align}
\end{subequations}
The objective in Eq.~\eqref{eq:dist_robust_portfolio_obj} is to maximize the expected rating sum under the worst-case probability measure, which is found by the inner minimization problem in Eq.~\eqref{eq:dist_robust_portfolio_obj} from the ambiguity set in Eq.~\eqref{eq:set_unc}. 

It is known that the maximin problem~\eqref{eq:dist_robust_portfolio} can be equivalently reformulated as a single-level maximization problem~\cite{DeYe10,KoTa23}. 
Specifically, deriving the dual maximization problem from the inner minimization problem in Eq.~\eqref{eq:dist_robust_portfolio_obj} gives the following mixed-integer semidefinite optimization (MISDO) problem: 
\begin{subequations}\label{eq:misdo}
\begin{align}
\underset{\substack{\bm{p}, \bm{P}, \bm{q}, \bm{Q}, r, s, \bm{x}}}{\text{maximize}} \quad & \Bigl(\hat{\bm{\mu}}_u \hat{\bm{\mu}}_u ^{\top} -\kappa_2 \hat{\bm{\Sigma}}_u \Bigr) \bullet \bm{Q}-\hat{\bm{\Sigma}}_u \bullet \bm{P}+2 \hat{\bm{\mu}}_u^{\top} \bm{p}-\kappa_1 r -s\label{eq:misdo_obj}\\
\text {subject to \hspace{1pt}} \quad & \bm{p}=-\bm{q} / 2-\bm{Q} \hat{\bm{\mu}}_u ,\label{eq:misdo_con1}\\
& \begin{pmatrix}
\bm{P} & \bm{p} \\
\bm{p}^{\top} & r
\end{pmatrix} \succeq \bm{O}, \quad 
\begin{pmatrix}
\bm{Q} & \bm{q} / 2+\bm{x} / 2 \\
(\bm{q} / 2+\bm{x} / 2)^{\top} & s
\end{pmatrix} \succeq \bm{O}, \label{eq:misdo_con2} \\
& \sum_{i \in I(u)} x_i=N, \label{eq:misdo_con3}\\
& \bm{x} \in\{0,1\}^{\lvert I(u) \rvert}, \label{eq:misdo_con4}
\end{align}
\end{subequations}
where $\bm{P}, \bm{Q} \in \mathbb{R}^{|I(u)| \times |I(u)|},~ \bm{p}$, $\bm{q} \in \mathbb{R}^{\lvert I(u) \rvert}$, and $r, s \in \mathbb{R}_{+}$ are dual decision variables of the inner minimization problem in Eq.~\eqref{eq:dist_robust_portfolio_obj}. 

\section{Penalty Alternating Direction Method}\label{sec:pen_alt}

In this section, we first formulate a penalized version of the MISDO problem~\eqref{eq:misdo}, and then present the PADM algorithm~\cite{geissler2017penalty,moreira2022alternating}, which efficiently finds high-quality solutions to the original MISDO problem~\eqref{eq:misdo} by repeatedly solving the penalized problem. 
We also discuss the convergence properties of the PADM algorithm. 

\subsection{Penalized Formulations}

We will now rewrite the MISDO problem~\eqref{eq:misdo} for the PADM algorithm design. 
Let us introduce a binary decision variable $\bm{z} \coloneqq (z_{i})_{i \in I(u)} \in \{0,1\}^{\lvert I(u) \rvert}$ as a copy of $\bm{x}$ (i.e., $\bm{x} = \bm{z}$). 
Next, we relax the binary constraint (Eq.~\eqref{eq:misdo_con4}) on $\bm{x}$ and move the cardinality constraint (Eq.~\eqref{eq:misdo_con3}) from $\bm{x}$ to $\bm{z}$.
As a result, we can equivalently rewrite the problem~\eqref{eq:misdo} as the following MISDO problem: 
\begin{subequations}\label{eq:misdo2}
\begin{align}
\underset{\substack{\bm{p}, \bm{P}, \bm{q}, \bm{Q}, r, s, \bm{x}, \bm{z}}}{\text{maximize}} \quad & \Bigl(\hat{\bm{\mu}}_u \hat{\bm{\mu}}_u ^{\top} -\kappa_2 \hat{\bm{\Sigma}}_u \Bigr) \bullet \bm{Q}-\hat{\bm{\Sigma}}_u \bullet \bm{P}+2 \hat{\bm{\mu}}_u^{\top} \bm{p}-\kappa_1 r - s \label{eq:misdo2_obj}\\
\text {subject to \hspace{1pt}} \quad 
& \bm{x} = \bm{z} \label{eq:misdo2_con1} \\
& \bm{p}=-\bm{q} / 2-\bm{Q} \hat{\bm{\mu}}_u ,\label{eq:misdo2_con2}\\
& \begin{pmatrix}
\bm{P} & \bm{p} \\
\bm{p}^{\top} & r
\end{pmatrix} \succeq \bm{O}, \quad 
\begin{pmatrix}
\bm{Q} & \bm{q} / 2+\bm{x} / 2 \\
(\bm{q} / 2+\bm{x} / 2)^{\top} & s
\end{pmatrix} \succeq \bm{O}, \label{eq:misdo2_con3} \\
& \bm{0} \le \bm{x} \le \bm{1}, \label{eq:misdo2_con4} \\
& \sum_{i \in I(u)} z_i=N, \quad \bm{z} \in\{0,1\}^{\lvert I(u) \rvert}. 
\label{eq:misdo2_con5}
\end{align}
\end{subequations}

\begin{lemma}\label{lemma:nonempty} \rm
The feasible set for problem~\eqref{eq:misdo2} is nonempty. 
\end{lemma}
\begin{proof}
For instance, a feasible solution can be constructed for the problem~\eqref{eq:misdo2} as follows:
\begin{equation}\label{eq:fsblsol}
\bm{p} = -\hat{\bm{\mu}}_u,~
\bm{q} = \bm{0},~
\bm{P} = \bm{Q} = \bm{I},~
r = \hat{\bm{\mu}}_u^{\top} \hat{\bm{\mu}}_u + 1,~
s = \frac{N}{4} + 1,~
\bm{x} = \bm{z} = \begin{pmatrix}
\bm{1} \\
\bm{0} 
\end{pmatrix},
\end{equation}
where $\bm{1} \in \mathbb{R}^N$ and $\bm{0} \in \mathbb{R}^{|I(u)|-N}$. 
Satisfaction of the positive semidefinite constraints in Eq.~\eqref{eq:misdo2_con3} can be verified from a property of the Schur complement~\cite{BoVa04}. 
\end{proof}

We now define a collection of the continuous decision variables as follows:
\[
\bm{X} \coloneqq (\bm{p}, \bm{P}, \bm{q}, \bm{Q}, r, s, \bm{x}). 
\]
We also denote the objective function in Eq.~\eqref{eq:misdo2_obj} as
\begin{align}
f(\bm{X},\bm{z}) \coloneqq \Bigl(\hat{\bm{\mu}}_u \hat{\bm{\mu}}_u ^{\top} -\kappa_2 \hat{\bm{\Sigma}}_u \Bigr) \bullet \bm{Q}-\hat{\bm{\Sigma}}_u \bullet \bm{P}+2 \hat{\bm{\mu}}_u^{\top} \bm{p} -\kappa_1 r - s, \label{eq:misdo_low_obj}
\end{align}
and the associated feasible sets as
\begin{align}
& \mathcal{X} \coloneqq \{\bm{X} = (\bm{p}, \bm{P}, \bm{q}, \bm{Q}, r, s, \bm{x}) \mid \mbox{Eqs.~\eqref{eq:misdo2_con2}--\eqref{eq:misdo2_con4}}\}, \label{eq:set_X} \\
& \mathcal{Z} \coloneqq \{\bm{z} \in\{0,1\}^{\lvert I(u) \rvert} \mid \sum_{i \in I(u)} z_i=N \}. \label{eq:set_z}
\end{align}
Then, the MISDO problem~\eqref{eq:misdo2} can be rewritten in the following simple form:
\begin{subequations}\label{eq:misdo3}
\begin{align}
\underset{\substack{\bm{X}, \bm{z}}}{\text{maximize}} \quad & f(\bm{X},\bm{z}) \label{eq:misdo3_obj}\\
\text {subject to} \quad 
& \bm{x} = \bm{z}, \label{eq:misdo3_con1} \\
& (\bm{X},\bm{z}) \in \mathcal{X} \times \mathcal{Z}. \label{eq:misdo3_con2}
\end{align}
\end{subequations}

We now relax the copy constraint in Eq.~\eqref{eq:misdo3_con1} and instead add the $L_1$-norm penalty term as 
\begin{subequations}\label{eq:misdo3_pen}
\begin{align}
\underset{\substack{\bm{X}, \bm{z}}}{\text{maximize}} \quad & f(\bm{X},\bm{z} \mid \gamma) \coloneqq f(\bm{X},\bm{z}) - \gamma \|\bm{x} -\bm{z}\|_1 \label{eq:misdo3_pen_obj}\\
\text {subject to} \quad 
& (\bm{X},\bm{z}) \in \mathcal{X} \times \mathcal{Z}, \label{eq:misdo3_pen_con1}
\end{align}
\end{subequations}
where $\gamma \in \mathbb{R}_{+}$ is a penalty weight parameter. 
We also define the partial optimality for this penalized problem~\eqref{eq:misdo3_pen}.  
\begin{definition}\label{def}\rm
Let $(\bm{X}^\star,\bm{z}^\star) \in \mathcal{X} \times \mathcal{Z}$ be a feasible solution to problem~\eqref{eq:misdo3_pen}. 
Then, $(\bm{X}^\star,\bm{z}^\star)$ is called a \emph{partial maximizer} to the problem~\eqref{eq:misdo3_pen} when it satisfies
\begin{align*}
& f(\bm{X}^\star,\bm{z}^\star \mid \gamma) \ge f(\bm{X},\bm{z}^\star \mid \gamma) \quad \mbox{for all $\bm{X} \in \mathcal{X}$}, \\
& f(\bm{X}^\star,\bm{z}^\star \mid \gamma) \ge f(\bm{X}^\star,\bm{z} \mid \gamma) \quad \mbox{for all $\bm{z} \in \mathcal{Z}$}.
\end{align*}
\end{definition}

\subsection{Algorithm}

Algorithm~\ref{alg:PADM} describes our PADM algorithm, which consists of a double-loop structure for efficiently finding high-quality solutions to the MISDO problem~\eqref{eq:misdo3}.
The inner loop solves the penalized problem~\eqref{eq:misdo3_pen} with a given penalty weight $\gamma$, by alternating between updating $\bm{X}$ and $\bm{z}$ (lines 4 and 5) until the update is smaller than an inner-loop threshold parameter $\varepsilon_{\text{in}} \in \mathbb{R}_{+}$ (line 6). 
The outer loop checks whether the $L_1$-distance between $\bm{x}$ and $\bm{z}$ is smaller than an outer-loop threshold parameter $\varepsilon_{\text{out}} \in \mathbb{R}_{+}$ (line 11). 
If this condition is met, we terminate the algorithm with the solution obtained so far (line 12); otherwise, we multiply the penalty weight $\gamma$ by a scaling parameter $\theta > 1$ (line 14) and return to the inner loop to solve the penalized problem~\eqref{eq:misdo3_pen} with the increased penalty weight $\gamma$. 

\begin{algorithm}
\caption{Penalty alternating direction method for solving problem~\eqref{eq:misdo3}}
\label{alg:PADM}
\textbf{Input:} $\varepsilon_{\text{in}} \in \mathbb{R}_{+}$, $\varepsilon_{\text{out}}  \in \mathbb{R}_{+}$, $\theta > 1$. \\
\textbf{Initialize:} $\bm{X}^{(0)} \in \mathcal{X}$, $\bm{z}^{(0)} \in \mathcal{Z}$, $\gamma^{(1)} \in \mathbb{R}_{+}$. 
\begin{algorithmic}[1]
\For{$t = 1, 2, \dots$} \Comment{outer loop}
    \State Set $\bm{X}^{(t,0)} \coloneqq \bm{X}^{(t-1)}$ and $\bm{z}^{(t,0)} \coloneqq \bm{z}^{(t-1)}$.
    \For{$k = 1, 2, \dots$} \Comment{inner loop}
        \State Find $\bm{X}^{(t,k)}$ by solving problem~\eqref{eq:misdo_x} with $(\gamma,\bar{\bm{z}}) = (\gamma^{(t)},\bm{z}^{(t,k-1)})$. 
        \State Compute $\bm{z}^{(t,k)}$ through Eq.~\eqref{eq:sort} from $\bar{\bm{x}} = \bm{x}^{(t,k)}$.
        \If{$\|(\bm{X}^{(t,k)}, \bm{z}^{(t,k)}) - (\bm{X}^{(t,k-1)}, \bm{z}^{(t,k-1)})\|_{\infty} \le \varepsilon_{\text{in}}$} 
            \State Set $(\bm{X}^{(t)}, \bm{z}^{(t)}) \coloneqq (\bm{X}^{(t,k)}, \bm{z}^{(t,k)})$.
            \State \textbf{break} \Comment{go to line 11}
        \EndIf
    \EndFor
    \If{$\|\bm{x}^{(t)} - \bm{z}^{(t)}\|_1 \leq \varepsilon_{\text{out}}$} \Comment{termination condition}
        \State Terminate the algorithm with $(\bm{X}^{(t)},\bm{z}^{(t)}) \in \mathcal{X} \times \mathcal{Z}$.
    \EndIf
    \State Update $\gamma^{(t+1)} \coloneqq \theta \gamma^{(t)}$. \Comment{penalty weight update}
\EndFor
\end{algorithmic}
\textbf{Output:} $(\bm{X}^{(t)},\bm{z}^{(t)}) \in \mathcal{X} \times \mathcal{Z}$.
\end{algorithm} 

The inner loop first updates the continuous decision variable $\bm{X} \in \mathcal{X}$ while keeping the binary decision variable $\bm{z} \in \mathcal{Z}$ fixed.
Specifically, we solve the following subproblem with $\bm{z} = \bar{\bm{z}}$:
\begin{subequations}\label{eq:misdo_x}
\begin{align}
\underset{\substack{\bm{X}}}{\text{maximize}} \quad & f(\bm{X},\bar{\bm{z}} \mid \gamma) \label{eq:misdo2_pen_obj}\\
\text {subject to} \quad 
& \bm{X} \in \mathcal{X}. \label{eq:misdo_x_con1}
\end{align}
\end{subequations}
This semidefinite optimization (SDO) problem, which does not contain integer decision variables, can be solved exactly in polynomial time by using optimization solvers based on interior-point methods~\cite{BoVa04}. 

The inner loop next updates the binary decision variable $\bm{z} \in \mathcal{Z}$ while keeping the continuous decision variable $\bm{X} \in \mathcal{X}$ fixed.
After omitting the parts that are irrelevant to this optimization, we solve the following subproblem with $\bm{x} = \bar{\bm{x}}$:
\begin{subequations}\label{eq:misdo_z}
\begin{align}
\underset{\substack{\bm{z}}}{\text{minimize}} \quad & \gamma \|\bar{\bm{x}} - \bm{z} \|_1 \label{eq:misdo_z_obj}\\
\text {subject to} \quad  
& \sum_{i \in I(u)} z_i=N, \quad \bm{z} \in\{0,1\}^{\lvert I(u) \rvert}. 
\label{eq:misdo_z_con1}
\end{align}
\end{subequations}
This problem can be solved easily by setting $z_i = 1$ for $i \in I(u)$ corresponding to the largest $N$ entries of $\bar{\bm{x}} \coloneqq (\bar{x}_i)_{i \in I(u)}$.
Specifically, we define the permutation $\sigma:~I(u) \to I(u)$ such that 
\[
1 \ge \bar{x}_{\sigma(1)} \ge \bar{x}_{\sigma(2)} \ge \cdots \ge \bar{x}_{\sigma(|I(u)|)} \ge 0.
\]
Then, we can find an optimal solution of $\bm{z} \in\{0,1\}^{\lvert I(u) \rvert}$ to problem~\eqref{eq:misdo_z} as follows:
\begin{equation}\label{eq:sort}
z_i \coloneqq
\begin{cases}
~1 & \mbox{if}~i \in \{\sigma(j) \mid j \in \{1,2,\ldots,N\}\}, \\
~0 & \mbox{otherwise}
\end{cases} \quad (i \in I(u)).
\end{equation}

\subsection{Convergence Properties}

\citet{geissler2017penalty} gave a general convergence theorem for the PADM algorithm under the assumption that the feasible sets $\mathcal{X}$ and $\mathcal{Z}$ are nonempty and compact. 
These feasible sets are guaranteed to be nonempty by Lemma~\ref{lemma:nonempty}, and $\mathcal{Z}$ is clearly compact by the definition of Eq.~\eqref{eq:set_z}, but $\mathcal{X}$ defined by Eq.~\eqref{eq:set_X} is not compact.
To overcome this challenge, we propose two revised formulations and discuss the convergence properties for those problems based on the existing result~\cite{geissler2017penalty}.

Before that, we provide the following lemma on the positive semidefinite constraint in Eq.~\eqref{eq:misdo2_con3}. 
\begin{lemma}\label{lemma:p_2^2} \rm
When $\begin{pmatrix}
        \bm{P} & \bm{p} \\
    \bm{    p}^{\top} & r
    \end{pmatrix} \succeq \bm{O}$, the following inequalities hold:
    \[
    \|\bm{p}\|_2^2 \le r \cdot \mathrm{tr}(\bm{P}), \quad 
    \|\bm{p}\|_2^2 \le r \|\bm{P}\|_{\rm F}.
    \]
\end{lemma}
\begin{proof}
Let $\lambda_i$ for each $i \in I(u)$ be an eigenvalue of the matrix $\bm{P}$, and $\lambda_{\max} \coloneqq \max_{i \in I(u)} \lambda_i$. 
Since $\mathrm{tr}(\bm{P}) = \sum_{i \in I(u)} \lambda_i \ge \lambda_{\max}$, we have $\bm{P} \preceq \lambda_{\max} \bm{I} \preceq \mathrm{tr}(\bm{P}) \bm{I}$.
From a property of the Schur complement~\cite{BoVa04}, it follows that 
\begin{align}
\begin{pmatrix}
\bm{P} & \bm{p} \\
\bm{p}^{\top} & r
\end{pmatrix} \succeq \bm{O} 
\quad & \Rightarrow \quad \begin{pmatrix}
\mathrm{tr}(\bm{P}) \bm{I} & \bm{p} \\
\bm{p}^{\top} & r
\end{pmatrix} \succeq \bm{O} \notag \\ 
& \Rightarrow \quad
r - \bm{p}^{\top} (\mathrm{tr}(\bm{P}) \bm{I})^{-1} \bm{p} \ge 0 \notag \\
& \Rightarrow \quad
\|\bm{p}\|_2^2 \le r \cdot \mathrm{tr}(\bm{P}). \notag
\end{align}
We can also prove that $\|\bm{p}\|_2^2 \le r \|\bm{P}\|_{\rm F}$ in the same manner because $\|\bm{P}\|_{\rm F} 
= \sqrt{\sum_{i \in I(u)} \lambda_i^2} \ge \lambda_{\max}$. 
\end{proof}

\subsubsection{Trace-constrained Formulation}

To guarantee the compactness of the feasible set, we here impose the upper-bound constraints on the traces of matrices in problem~\eqref{eq:misdo3}: 
\begin{subequations}\label{eq:misdo_tr}
\begin{align}
\underset{\substack{\bm{X}, \bm{z}}}{\text{maximize}} \quad & f(\bm{X},\bm{z}) \label{eq:misdo_tr_obj}\\
\text {subject to} \quad 
& \bm{x} = \bm{z}, \label{eq:misdo_tr_con1} \\
& (\bm{X},\bm{z}) \in \mathcal{X} \times \mathcal{Z}, \label{eq:misdo_tr_con2} \\
& \mathrm{tr}(\bm{P}) \le \tau_P, \quad \mathrm{tr}(\bm{Q}) \le \tau_Q, \label{eq:misdo_tr_con3}
\end{align}
\end{subequations}
where $\tau_P,\tau_Q \in \mathbb{R}_{+}$ are trace upper-bound parameters. 

Let $\bar{f}_1 \in \mathbb{R}$ be a lower bound on the optimal objective value of problem~\eqref{eq:misdo_tr}; this lower bound can be obtained from any feasible solution to the problem~\eqref{eq:misdo_tr}. 
Therefore, the following relationship holds for any optimal solution $(\bm{X},\bm{z})$ to this problem: 
\begin{align}\label{eq:lb1}
\bar{f}_1 \le f(\bm{X},\bm{z}). 
\end{align}

We now define the trace-constrained feasible set by imposing Eqs.~\eqref{eq:misdo_tr_con3} and \eqref{eq:lb1} on the original feasible set $\mathcal{X}$ of Eq.~\eqref{eq:set_X}; that is,
\[
\mathcal{X}_1 \coloneqq \{\bm{X} = (\bm{p}, \bm{P}, \bm{q}, \bm{Q}, r, s, \bm{x}) \mid \mbox{Eq.~\eqref{eq:def_Q1}}\},
\]
where
\begin{subequations}\label{eq:def_Q1}
\begin{align}
& \bm{p}=-\bm{q} / 2-\bm{Q} \hat{\bm{\mu}}_u ,\label{eq:def_Q1_con1} \\
& \begin{pmatrix}
\bm{P} & \bm{p} \\
\bm{p}^{\top} & r
\end{pmatrix} \succeq \bm{O}, \quad 
\begin{pmatrix}
\bm{Q} & \bm{q} / 2+\bm{x} / 2 \\
(\bm{q} / 2+\bm{x} / 2)^{\top} & s
\end{pmatrix} \succeq \bm{O}, \label{eq:def_Q1_con2} \\
& \bm{0} \le \bm{x} \le \bm{1}, \label{eq:def_Q1_con3} \\
& \mathrm{tr}(\bm{P}) \le \tau_P, \quad \mathrm{tr}(\bm{Q}) \le \tau_Q, \label{eq:def_Q1_con4} \\
& \bar{f}_1 \le \Bigl(\hat{\bm{\mu}}_u \hat{\bm{\mu}}_u ^{\top} -\kappa_2 \hat{\bm{\Sigma}}_u \Bigr) \bullet \bm{Q} -\hat{\bm{\Sigma}}_u \bullet \bm{P} + 2 \hat{\bm{\mu}}_u^{\top} \bm{p} -\kappa_1 r - s. \label{eq:def_Q1_con5} 
\end{align}
\end{subequations}

The following proposition states the compactness of the trace-constrained feasible set $\mathcal{X}_1$. 

\begin{proposition}\label{prop:compact2} \rm 
When $\kappa_1 > 0$, the feasible set $\mathcal{X}_1$ defined by Eq.~\eqref{eq:def_Q1} is compact. 
\end{proposition}
\begin{proof}
Let us suppose that $(\bm{p}, \bm{P}, \bm{q}, \bm{Q}, r, s, \bm{x}) \in \mathcal{X}_1$. 
Since $\bm{P} \coloneqq (p_{ij})_{(i,j) \in I(u) \times I(u)}$ is positive semidefinite from Eq.~\eqref{eq:def_Q1_con2}, its principal submatrices are also positive semidefinite, and thus, $p_{ij}^2 \le p_{ii} p_{jj} \le \tau_P^2$ from Eq.~\eqref{eq:def_Q1_con4} for all $(i,j) \in I(u) \times I(u)$. 
The same relationship holds for $\bm{Q} \succeq \bm{O}$, and therefore, the absolute values of all entries of $\bm{P}$ and $ \bm{Q}$ are bounded by $\tau_P$ and $\tau_Q$, respectively. 

Note that
\begin{align}
& \bar{f}_1 - \Bigl(\hat{\bm{\mu}}_u \hat{\bm{\mu}}_u ^{\top} -\kappa_2 \hat{\bm{\Sigma}}_u \Bigr) \bullet \bm{Q} + \hat{\bm{\Sigma}}_u \bullet \bm{P} \notag \\
\le~ & 2 \hat{\bm{\mu}}_u^{\top} \bm{p} -\kappa_1 r - s \quad \because \mbox{Eq.~\eqref{eq:def_Q1_con5}} \notag \\
\le~ & 2 \|\hat{\bm{\mu}}_u\|_2 \|\bm{p}\|_2 -\kappa_1 r - s \quad \because \mbox{Cauchy--Schwarz inequality} \notag \\
\le~ & 2 \|\hat{\bm{\mu}}_u\|_2 \sqrt{r \cdot \mathrm{tr}(\bm{P})} -\kappa_1 r - s \quad \because \mbox{Lemma~\ref{lemma:p_2^2}} \notag \\
\le~ & - \kappa_1 \left( \sqrt{r} - \frac{\|\hat{\bm{\mu}}_u\|_2 \sqrt{\mathrm{tr}(\bm{P})}}{\kappa_1} \right)^2 + \frac{\|\hat{\bm{\mu}}_u\|_2^2 \cdot \mathrm{tr}(\bm{P})}{\kappa_1} - s. \quad \because \mbox{$\kappa_1 > 0$}
\notag
\end{align}
It then follows that 
\begin{align}
& \kappa_1 \left( \sqrt{r} - \frac{\|\hat{\bm{\mu}}_u\|_2 \sqrt{\mathrm{tr}(\bm{P})}}{\kappa_1} \right)^2 +s \notag \\
\le~ & - \bar{f}_1 + \Bigl(\hat{\bm{\mu}}_u \hat{\bm{\mu}}_u ^{\top} -\kappa_2 \hat{\bm{\Sigma}}_u \Bigr) \bullet \bm{Q} - \hat{\bm{\Sigma}}_u \bullet \bm{P} + \frac{\|\hat{\bm{\mu}}_u\|_2^2 \cdot \mathrm{tr}(\bm{P})}{\kappa_1}, \notag 
\end{align}
which implies that $r, s \in \mathbb{R}_{+}$ are each contained in a bounded set. 
Then, Lemma~\ref{lemma:p_2^2} ensures that $\bm{p}, \bm{q} \in \mathbb{R}^{|I(u)|}$ are each contained in a bounded set. 
The proof is therefore completed by noting that $\bm{x} \in \mathbb{R}^{|I(u)|}$ is bounded by Eq.~\eqref{eq:def_Q1_con3}. 
\end{proof}

We now obtain the following convergence theorem of Algorithm~\ref{alg:PADM} for solving the trace-constrained problem~\eqref{eq:misdo_tr}. 
\begin{theorem}\label{thm_convergence2}\rm
Suppose that $\kappa_1 > 0$ and that $\tau_P,\tau_Q \in \mathbb{R}_{+}$ are sufficiently large numbers. 
Also suppose that $(\bm{X}^{(t)},\bm{z}^{(t)}) \to (\hat{\bm{X}},\hat{\bm{z}})$ for $t \to +\infty$, where $(\bm{X}^{(t)},\bm{z}^{(t)})$ is a partial maximizer to problem~\eqref{eq:misdo3_pen} with $\gamma = \gamma^{(t)}$ generated by Algorithm~\ref{alg:PADM} with $\mathcal{X} = \mathcal{X}_1$.
Then, $(\hat{\bm{X}},\hat{\bm{z}})$ is a stationary point to problem~\eqref{eq:misdo_tr}.     
\end{theorem}
\begin{proof}
When $\tau_P,\tau_Q \in \mathbb{R}_{+}$ are sufficiently large (e.g., $\tau_P,\tau_Q \ge |I(u)|$), Lemma~\ref{lemma:nonempty} and Proposition~\ref{prop:compact2} ensure that the feasible sets $\mathcal{X}_1$ and $\mathcal{Z}$ are nonempty and compact. 
Since the objective function of problem~\eqref{eq:misdo_tr} (i.e., Eq.~\eqref{eq:misdo_low_obj}) is continuously differentiable, the convergence follows from Theorem 8b \cite{geissler2017penalty}. 
\end{proof}

Note that this theorem shows the property of the limit point upon convergence, but does not guarantee the convergence of the solution sequence or the feasibility of limit point; however, \citet{moreira2022alternating} reported that they never encountered such issues in their numerical experiments. 

\subsubsection{Frobenius-norm-regularized Formulation}

To guarantee the compactness of the feasible set, we here add the Frobenius norm regularization terms to problem~\eqref{eq:misdo3}: 
\begin{subequations}\label{eq:misdo_reg}
\begin{align}
\underset{\substack{\bm{X}, \bm{z}}}{\text{maximize}} \quad & f(\bm{X},\bm{z}) - \lambda_P \|\bm{P}\|_{\mathrm{F}}^2 - \lambda_Q \|\bm{Q}\|_{\mathrm{F}}^2 \label{eq:misdo_reg_obj}\\
\text {subject to} \quad 
& \bm{x} = \bm{z}, \label{eq:misdo_reg_con1} \\
& (\bm{X},\bm{z}) \in \mathcal{X} \times \mathcal{Z}, \label{eq:misdo_reg_con2} 
\end{align}
\end{subequations}
where $\lambda_P,\lambda_Q \in \mathbb{R}_{+}$ are regularization weight parameters.

Let $\bar{f}_2 \in \mathbb{R}$ be a lower bound on the optimal objective value of problem~\eqref{eq:misdo_reg}; this lower bound can be obtained from any feasible solution to the problem~\eqref{eq:misdo_reg}. 
Therefore, the following relationship holds for any optimal solution $(\bm{X},\bm{z})$ to this problem: 
\begin{align}\label{eq:lb2}
\bar{f}_2 \le f(\bm{X},\bm{z}) - \lambda_P \|\bm{P}\|_{\mathrm{F}}^2 - \lambda_Q \|\bm{Q}\|_{\mathrm{F}}^2. 
\end{align}

We now define the Frobenius-norm-regularized feasible set by imposing Eq.~\eqref{eq:lb2} on the original feasible set $\mathcal{X}$ of Eq.~\eqref{eq:set_X}; that is, 
\[
\mathcal{X}_2 \coloneqq \{\bm{X} = (\bm{p}, \bm{P}, \bm{q}, \bm{Q}, r, s, \bm{x}) \mid \mbox{Eq.~\eqref{eq:def_Q2}}\},
\]
where
\begin{subequations}\label{eq:def_Q2}
\begin{align}
& \bm{p}=-\bm{q} / 2-\bm{Q} \hat{\bm{\mu}}_u ,\label{eq:def_Q2_con1} \\
& \begin{pmatrix}
\bm{P} & \bm{p} \\
\bm{p}^{\top} & r
\end{pmatrix} \succeq \bm{O}, \quad 
\begin{pmatrix}
\bm{Q} & \bm{q} / 2+\bm{x} / 2 \\
(\bm{q} / 2+\bm{x} / 2)^{\top} & s
\end{pmatrix} \succeq \bm{O}, \label{eq:def_Q2_con2} \\
& \bm{0} \le \bm{x} \le \bm{1}, \label{eq:def_Q2_con3} \\
& \bar{f}_2 \le \Bigl(\hat{\bm{\mu}}_u \hat{\bm{\mu}}_u ^{\top} -\kappa_2 \hat{\bm{\Sigma}}_u \Bigr) \bullet \bm{Q} -\hat{\bm{\Sigma}}_u \bullet \bm{P} + 2 \hat{\bm{\mu}}_u^{\top} \bm{p} -\kappa_1 r - s - \lambda_P \|\bm{P}\|_{\mathrm{F}}^2 - \lambda_Q \|\bm{Q}\|_{\mathrm{F}}^2. \label{eq:def_Q2_con4} 
\end{align}
\end{subequations}

The following proposition states the compactness of the Frobenius-norm-regularized feasible set $\mathcal{X}_2$. 
\begin{proposition}\label{prop:compact} \rm
Suppose that $\kappa_1 > 0$, $\lambda_P > 0$, and $\lambda_Q > 0$. 
Then, the feasible set $\mathcal{X}_2$ defined by Eq.~\eqref{eq:def_Q2} is compact. 
\end{proposition}
\begin{proof}
Let us suppose that $(\bm{p}, \bm{P}, \bm{q}, \bm{Q}, r, s, \bm{x}) \in \mathcal{X}_2$. 
Note that 
\begin{align}
& \bar{f}_2 - \Bigl(\hat{\bm{\mu}}_u \hat{\bm{\mu}}_u ^{\top} -\kappa_2 \hat{\bm{\Sigma}}_u \Bigr) \bullet \bm{Q} + \hat{\bm{\Sigma}}_u \bullet \bm{P} + \lambda_P \|\bm{P}\|_{\mathrm{F}}^2 + \lambda_Q \|\bm{Q}\|_{\mathrm{F}}^2 \notag \\
\le~ & 2 \hat{\bm{\mu}}_u^{\top} \bm{p} -\kappa_1 r - s \quad \because \mbox{Eq.~\eqref{eq:def_Q2_con4}} \notag \\
\le~ & 2 \|\hat{\bm{\mu}}_u\|_2 \|\bm{p}\|_2 -\kappa_1 r - s \quad \because \mbox{Cauchy--Schwarz inequality} \notag \\
\le~ & 2 \|\hat{\bm{\mu}}_u\|_2 \sqrt{r \|\bm{P}\|_{\rm F}} -\kappa_1 r - s \quad \because \mbox{Lemma~\ref{lemma:p_2^2}} \notag \\
\le~ & - \kappa_1 \left( \sqrt{r} - \frac{\|\hat{\bm{\mu}}_u\|_2 \sqrt{\|\bm{P}\|_{\rm F}}}{\kappa_1} \right)^2 + \frac{\|\hat{\bm{\mu}}_u\|_2^2 \|\bm{P}\|_{\rm F}}{\kappa_1} - s \quad \because \mbox{$\kappa_1 > 0$} \label{eq:rs} \\
\le~ & \frac{\|\hat{\bm{\mu}}_u\|_2^2 \|\bm{P}\|_{\rm F}}{\kappa_1}. \quad \because \mbox{$\kappa_1, s \ge 0$} \label{eq:muPkappa}
\end{align}
It follows from Eq.~\eqref{eq:muPkappa} that
\begin{align}
& \lambda_P \|\bm{P}\|_{\mathrm{F}}^2 + \lambda_Q \|\bm{Q}\|_{\mathrm{F}}^2 \notag \\
\le~ & -\bar{f}_2 + \Bigl(\hat{\bm{\mu}}_u \hat{\bm{\mu}}_u ^{\top} -\kappa_2 \hat{\bm{\Sigma}}_u \Bigr) \bullet \bm{Q} -\hat{\bm{\Sigma}}_u \bullet \bm{P} + \frac{\|\hat{\bm{\mu}}_u\|_2^2 \|\bm{P}\|_{\rm F}}{\kappa_1} \notag \\
\le~ & -\bar{f}_2 + \| \hat{\bm{\mu}}_u \hat{\bm{\mu}}_u ^{\top} -\kappa_2 \hat{\bm{\Sigma}}_u \|_{\rm F} \|\bm{Q}\|_{\rm F} + \|\hat{\bm{\Sigma}}_u\|_{\rm F} \|\bm{P}\|_{\rm F} + \frac{\|\hat{\bm{\mu}}_u\|_2^2 \|\bm{P}\|_{\rm F}}{\kappa_1}. \\
& \quad \because \mbox{Cauchy--Schwarz inequality} \notag
\end{align}
Since $\lambda_P > 0$ and $\lambda_Q > 0$, matrices $\bm{P}, \bm{Q} \in \mathbb{R}^{|I(u)| \times |I(u)|}$ are each contained in a bounded set.

Moreover, it follows from Eq.~\eqref{eq:rs} that 
\begin{align}
\kappa_1 \left( \sqrt{r} - \frac{\|\hat{\bm{\mu}}_u\|_2 \sqrt{\|\bm{P}\|_{\rm F}}}{\kappa_1} \right)^2 +s 
\le~ & - \bar{f}_2 + \Bigl(\hat{\bm{\mu}}_u \hat{\bm{\mu}}_u ^{\top} -\kappa_2 \hat{\bm{\Sigma}}_u \Bigr) \bullet \bm{Q} - \hat{\bm{\Sigma}}_u \bullet \bm{P} \notag \\
& - \lambda_P \|\bm{P}\|_{\mathrm{F}}^2 - \lambda_Q \|\bm{Q}\|_{\mathrm{F}}^2 
+ \frac{\|\hat{\bm{\mu}}_u\|_2^2 \|\bm{P}\|_{\rm F}}{\kappa_1}, \notag 
\end{align}
which implies that $r, s \in \mathbb{R}_{+}$ are each contained in a bounded set. 
Then, Lemma~\ref{lemma:p_2^2} ensures that $\bm{p}, \bm{q} \in \mathbb{R}^{|I(u)|}$ are each contained in a bounded set. 
The proof is therefore completed by noting that $\bm{x} \in \mathbb{R}^{|I(u)|}$ is bounded by Eq.~\eqref{eq:def_Q2_con3}. 
\end{proof}

We now obtain the following convergence theorem of Algorithm~\ref{alg:PADM} for solving the Frobenius-norm-regularized problem~\eqref{eq:misdo_reg}. 
\begin{theorem}\label{thm_convergence1}\rm
Suppose that $\kappa_1 > 0$, $\lambda_P > 0$, and $\lambda_Q > 0$. 
Also suppose that $(\bm{X}^{(t)},\bm{z}^{(t)}) \to (\hat{\bm{X}},\hat{\bm{z}})$ for $t \to +\infty$, where $(\bm{X}^{(t)},\bm{z}^{(t)})$ is a partial maximizer to problem~\eqref{eq:misdo3_pen} with $\gamma = \gamma^{(t)}$ generated by Algorithm~\ref{alg:PADM} with $\mathcal{X} = \mathcal{X}_2$.
Then, $(\hat{\bm{X}},\hat{\bm{z}})$ is a stationary point to problem~\eqref{eq:misdo_reg}.     
\end{theorem}
\begin{proof}
Lemma~\ref{lemma:nonempty} and Proposition~\ref{prop:compact} ensure that the feasible sets $\mathcal{X}_2$ and $\mathcal{Z}$ are nonempty and compact. 
Since the objective function of the problem~\eqref{eq:misdo_reg} (cf.~Eq.~\eqref{eq:misdo_low_obj}) is continuously differentiable, the convergence follows from Theorem 8b~\cite{geissler2017penalty}. 
\end{proof}

\section{Experimental Results}\label{sec:exp_res}

In this section, we evaluate the effectiveness of our recommendation selection method through computational experiments using three publicly available datasets.
All experiments were performed on a macOS 14.4 computer equipped with an Apple M3 processor (8 cores, 4.05 GHz) and 24 GB of RAM.

\subsection{Rating Datasets}\label{subsec:datasets}
We used three publicly available rating datasets widely used in recommender systems research: the MovieLens\footnote[1]{\url{https://grouplens.org/datasets/movielens/100k/}} (100K), Yahoo! R3\footnote[2]{\url{https://webscope.sandbox.yahoo.com/}} (Yahoo! music ratings for user selected and randomly selected songs, v.~1.0), and BookCrossing\footnote[3]{\url{https://grouplens.org/datasets/book-crossing/}} datasets. 
\begin{itemize}
\item \textbf{MovieLens}: This dataset contains 100,000 ratings on a 5-point scale (i.e., $1,2,\ldots,5$) provided by 943 users for 1,682 movies. 
For our experiments, we extracted 568 users who provided at least 50 ratings.
\vspace{1mm}
\item \textbf{Yahoo! R3}: This dataset was pre-split into training and testing sets containing 15,400 and 5,400 users, respectively, with each user rating a subset of the 1,000 songs on a 5-point scale (i.e., $1,2,\ldots,5$).
For our experiments, we extracted 5,050 users who provided at least 20 ratings in the training set.
\vspace{1mm}
\item \textbf{BookCrossing}: The dataset contains 1,149,780 ratings on an 11-point scale (i.e., $0,1,\ldots,10$) provided by 278,858 users for 271,376 books.
For our experiments, we extracted 3,156 users who provided at least 20 ratings.
\end{itemize}

For these datasets, each user's ratings were randomly split into a training set (60\%) for selecting recommendations and a testing set (40\%) for assessing the quality of those recommendations. 
Recommendations were selected for 100 randomly sampled target users, and the results were averaged over 10 independent runs. 

\subsection{Recommendation Selection Methods}\label{subsec:methods_comp}

We evaluated the performance of the following three methods for recommendation selection:
\begin{itemize}
    \item \textbf{MV}: Mean--variance portfolio optimization model~\cite{Wa09,HuZh11,YaYa26};
    \vspace{1mm}
    \item \textbf{RMV($\Gamma^{(\mu)}, \Gamma^{(\sigma)}$)}: Robust mean--variance portfolio optimization model~\cite{YaIk24};
    \vspace{1mm}
    \item \textbf{DRO}: Our DRO model~\eqref{eq:misdo}, which was solved using the PADM algorithm (Algorithm~\ref{alg:PADM});
\end{itemize}
where $\Gamma^{(\mu)}, \Gamma^{(\sigma)} \in \mathbb{Z}_{+}$ are the cardinality parameters of the uncertainty sets for the expectation and the covariance of ratings, respectively; see Yanagi et al.~\cite{YaIk24} for details on the uncertainty sets. 

For rating prediction, we used singular value decomposition from the Python \texttt{Surprise} library\footnote[4]{\url{http://surpriselib.com/}}~\cite{Hug20}, where the hyperparameters were set to 100 factors, a learning rate of 0.01, and a regularization weight of 0.1.
For each user $u \in U$, the candidate item set $I(u)$ was composed of all the items rated by the user in the corresponding testing set. 
The number of recommended items was set as $N \in \{3, 5\}$.
To stabilize the estimation of the rating covariance matrix, we applied the shrinkage estimation method~\cite{YaYa26} for each user $u \in U$ as
\[
\hat{\bm{\Sigma}}_u = 0.25\bm{S}_u + 0.75\bm{F}_u,
\]
where $\bm{S}_u \in \mathbb{R}^{\lvert I(u) \rvert \times \lvert I(u) \rvert}$ is the sample covariance matrix, and $\bm{F}_u \in \mathbb{R}^{\lvert I(u) \rvert \times \lvert I(u) \rvert}$ is a target matrix based on matrix completion; see Yanagi et al.~\cite{YaYa26} for technical details on the shrinkage estimation.

For the MV and RMV methods, we used Gurobi Optimizer\footnote[5]{\url{https://www.gurobi.com/}} 11.0.3 to solve the corresponding portfolio optimization problems, with the computation time limited to 5 s per user. 
These objective functions are defined as a weighted sum of the mean and variance of the user rating sum (i.e., $\max\{(1 - \alpha) \cdot \mathrm{mean} - \alpha \cdot \mathrm{variance}\}$), where the risk aversion parameter was varied as $\alpha \in \{0.0, 0.1, \ldots, 0.5\}$.

For the RMV method, we defined the uncertainty ranges for the expectation and covariance of ratings as
\begin{align}
& \delta_{ui}^{(\mu)} = \frac{\sqrt{\hat{\sigma}_{ii}}}{\sqrt{\max\{1, n_i\}}} \quad ((u,i) \in U \times I), \quad 
\delta_{ij}^{(\sigma)} = \frac{0.2 \cdot \hat{\sigma}_{ij}}{\sqrt{\max\{1,n_{ij}\}}} \quad ((i,j) \in I \times I), \notag
\end{align}
where $n_i$ is the number of users who rated item $i \in I$, and $n_{ij}$ is the number of users who rated both items $i,j \in I$.
We also set the cardinality parameters for the expectation and the covariance as $(\Gamma^{(\mu)}, \Gamma^{(\sigma)}) \in \{(2, 25), (4, 50)\}$; see Yanagi et al.~\cite{YaIk24} for further details on the uncertainty sets.

For our DRO method, we set $\kappa_1,\kappa_2 \in \{0.1, 1.0, 5.0\}$ as the ambiguity parameters in Eq.~\eqref{eq:set_unc}. 
In Algorithm \ref{alg:PADM}, we solved the Frobenius-norm-regularized problem~\eqref{eq:misdo_reg}, with the regularization weight parameters $\lambda_P = \lambda_Q = 10$. 
We used MOSEK\footnote[6]{\url{https://www.mosek.com/}} 10.2.2 to solve the SDO subproblems. 
We set the convergence threshold parameters $\varepsilon_{\text{in}} = \varepsilon_{\text{out}} = 10^{-4}$, the penalty weight scaling parameter $\theta = 10$, and the initial penalty weight $\gamma^{(1)} = 10^{-4}$.
For the initial solutions ${\bm x}^{(0)}$ and ${\bm z}^{(0)}$, we set the entries corresponding to the top ten items with predicted ratings to $N/10$, and all the remaining entries to 0.

\subsection{Evaluation Metrics}\label{subsec:eval}

To evaluate the recommendation quality in terms of accuracy and diversity for testing sets, we used the following two metrics~\cite{ShGu11}: 
\begin{itemize}
\item \textbf{F1 score}: Average F1 score of recommendations across target users; 
\vspace{1mm}
\item $\bf 1 - \mbox{\bf Gini}$: One minus the Gini coefficient, which was calculated for the number of recommendations across items. 
\end{itemize}

The F1 score measured recommendation accuracy based on the intersection between recommended items and highly rated items.
Here, highly rated items were defined as items with a rating of 4 or higher in the MovieLens dataset, a rating of 3 or higher in the Yahoo! R3 dataset, and a rating of 7 or higher in the BookCrossing dataset.
The higher the F1 score, the more accurate the recommendations.

The Gini coefficient measured the inequality in the number of recommendations among items.
In other words, 
the higher the value of ``$1 - \mbox{Gini}$,'' the more diverse the recommendations.

\begin{figure}[!tbh]
    \centering
    \begin{subfigure}{0.49\linewidth}
        \centering
        \includegraphics[width=\linewidth]{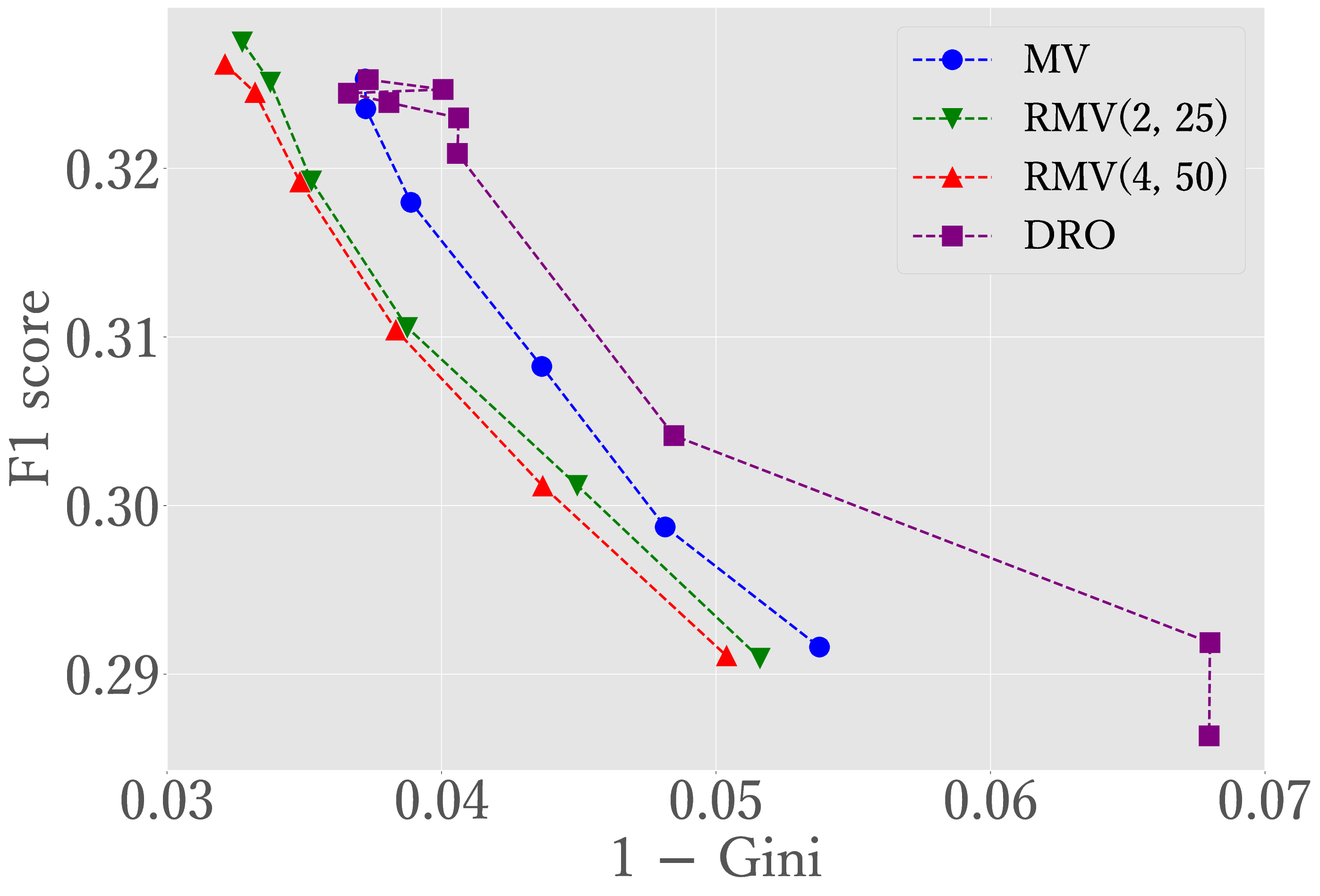}
        \caption{MovieLens dataset ($N=3$)}
        \label{fig:movielens_3}
    \end{subfigure}
    \begin{subfigure}{0.49\linewidth}
        \centering
        \includegraphics[width=\linewidth]{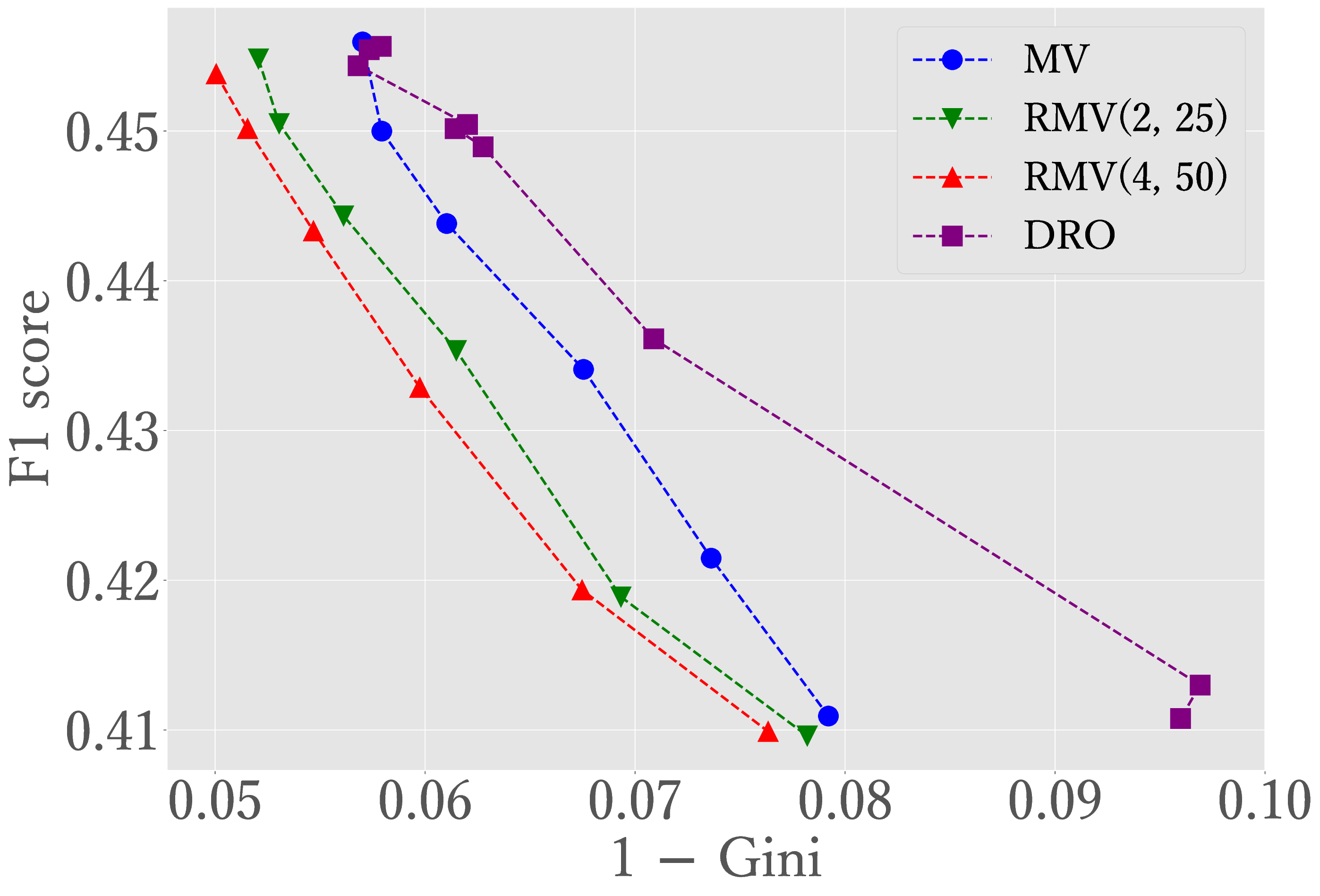}
        \caption{MovieLens dataset ($N=5$)}
        \label{fig:movielens_5}
    \end{subfigure}

    \begin{subfigure}{0.49\linewidth}
        \centering
        \includegraphics[width=\linewidth]{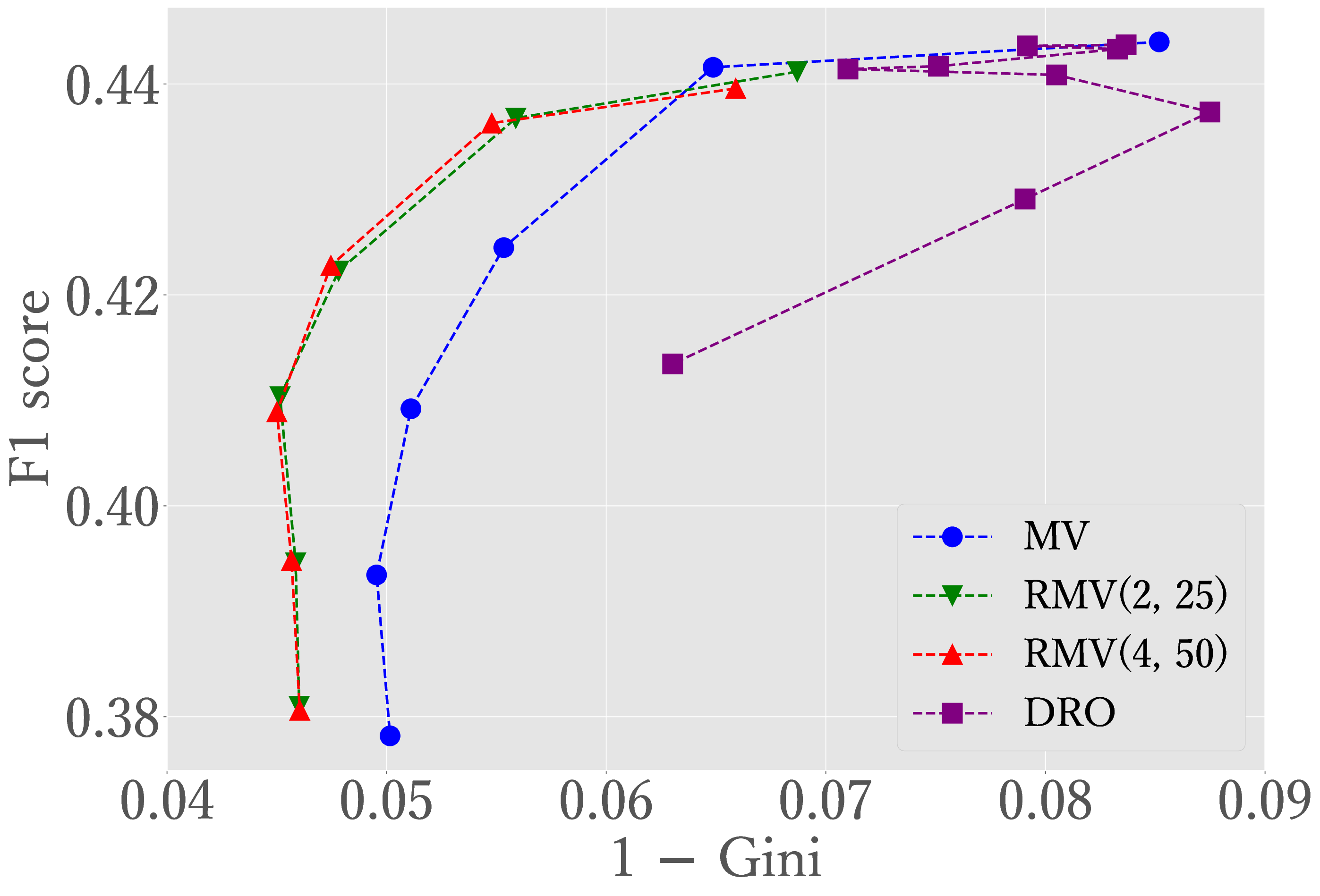}
        \caption{Yahoo! R3 dataset ($N=3$)}
        \label{fig:r3_3}
    \end{subfigure}
    \medskip
    \begin{subfigure}{0.49\linewidth}
        \centering
        \includegraphics[width=\linewidth]{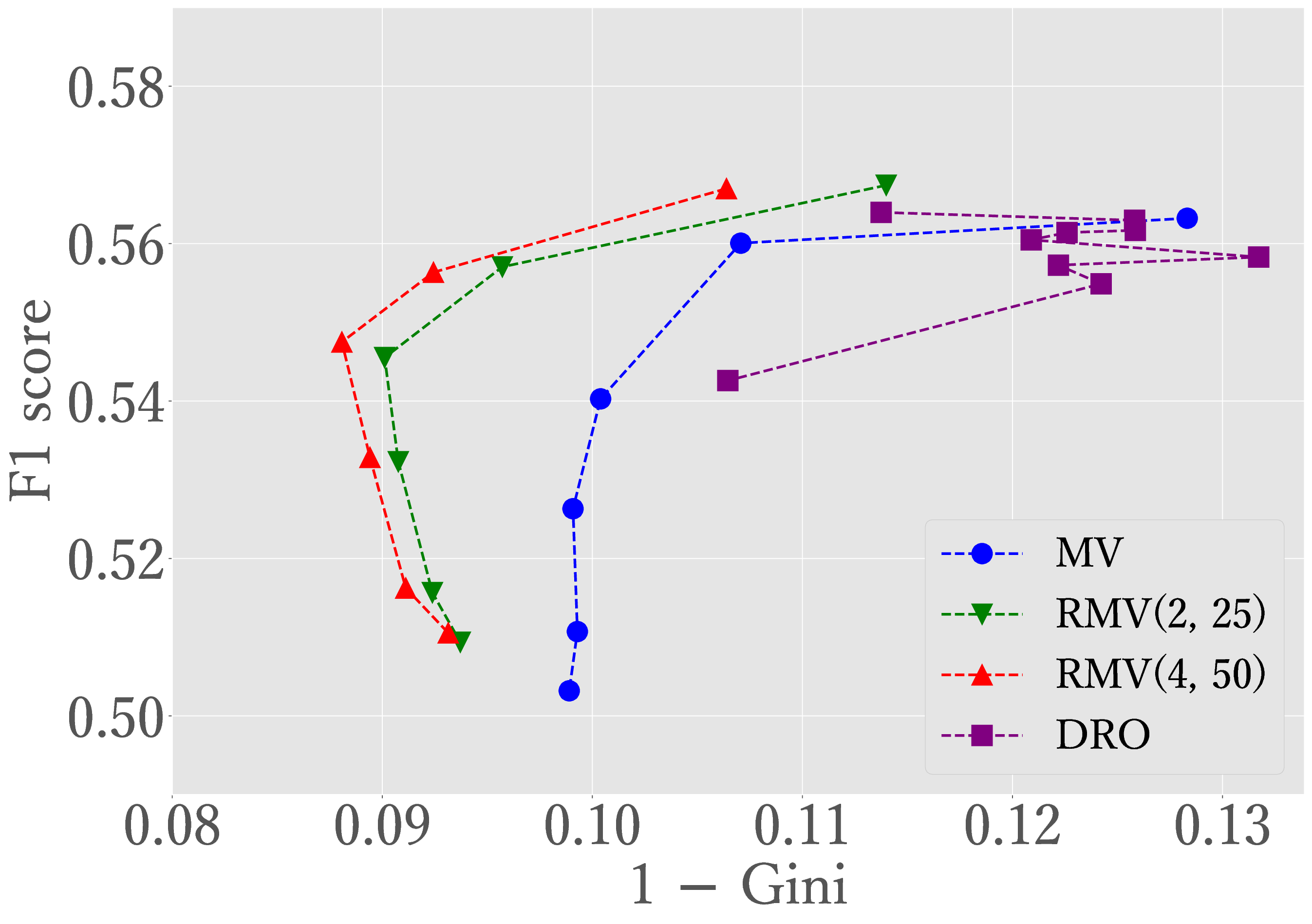}
        \caption{Yahoo! R3 dataset ($N=5$)}
        \label{fig:r3_5}
    \end{subfigure}

    \begin{subfigure}{0.49\linewidth}
        \centering
        \includegraphics[width=\linewidth]{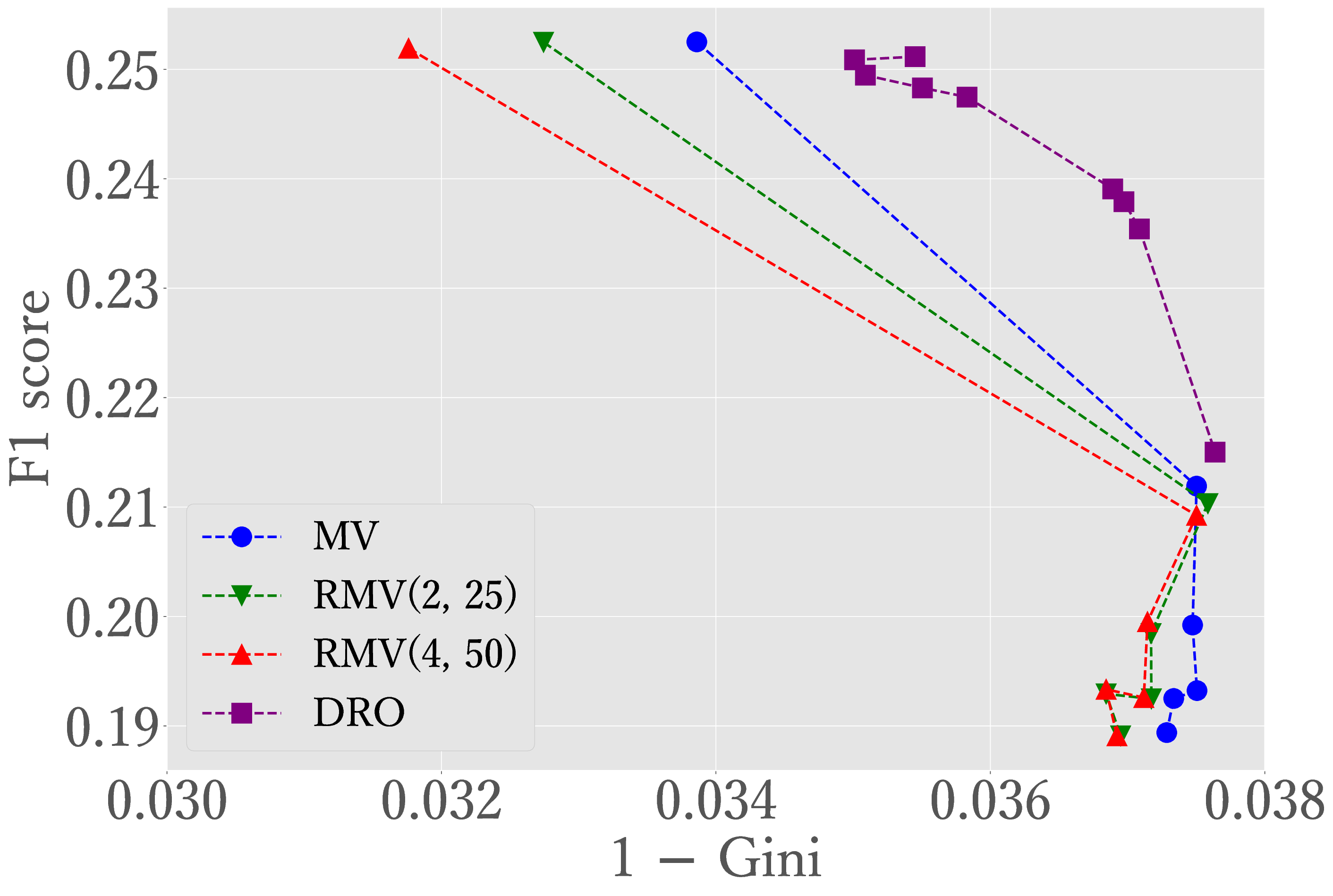}
        \caption{BookCrossing dataset ($N=3$)}
        \label{fig:book_3}
    \end{subfigure}
    \begin{subfigure}{0.49\linewidth}
        \centering
        \includegraphics[width=\linewidth]{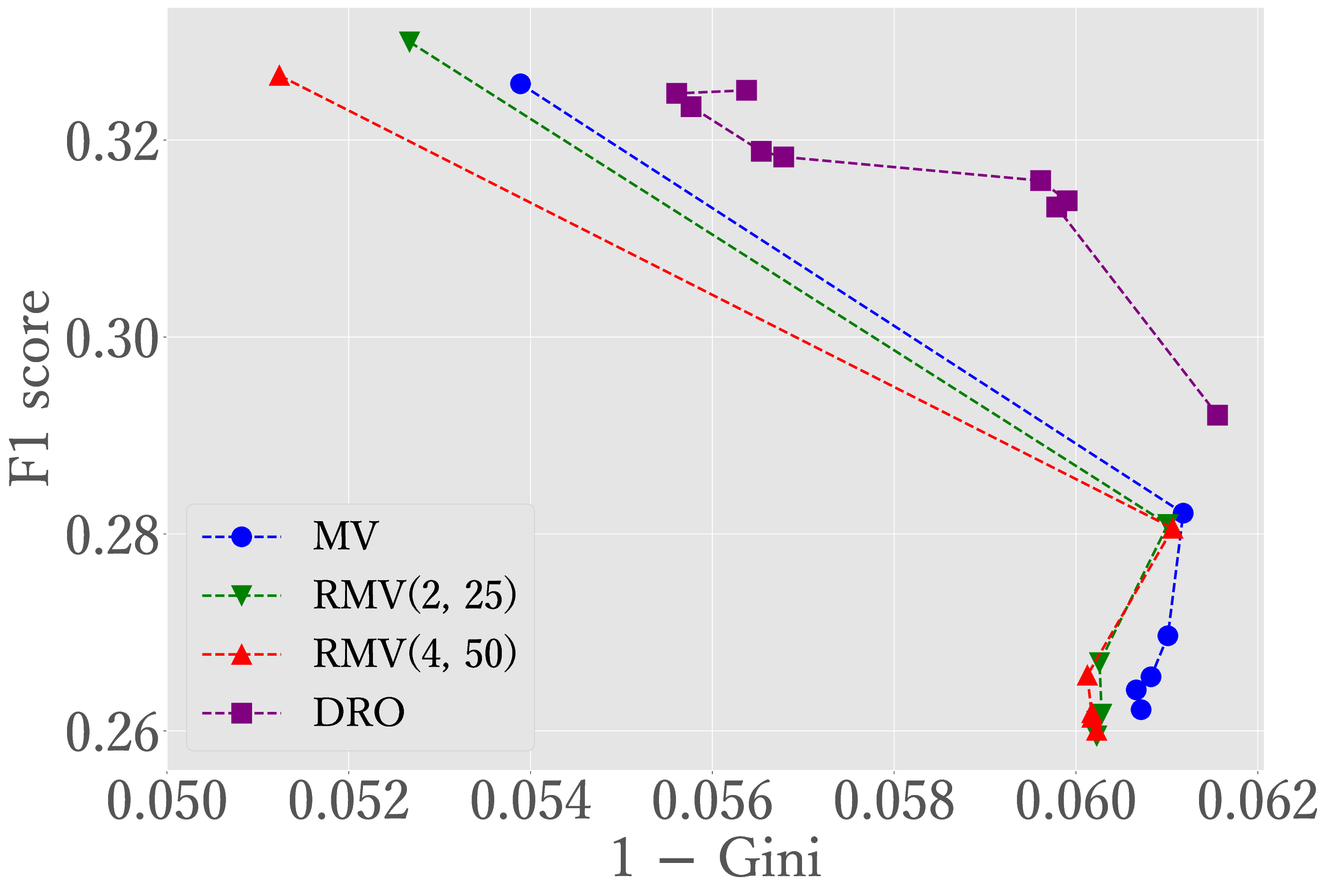}
        \caption{BookCrossing dataset ($N=5$)}
        \label{fig:book_5}
    \end{subfigure}

    \caption{Recommendation quality with the number
of recommended items $N \in \{3, 5\}$}
    \label{fig:result}
\end{figure}

\subsection{Results of Recommendation Quality}\label{subsec:result_rec}

Figure~\ref{fig:result} shows the recommendation quality on the three datasets with the number of recommended items $N \in \{3, 5\}$.
The vertical axis represents the F1 score as a measure of recommendation accuracy, while the horizontal axis represents the value of ``$1 - \mbox{Gini}$'' as a measure of recommendation diversity. 
Therefore, data points located at the top indicate more accurate recommendations, while those located to the right indicate more diverse recommendations.

For the MovieLens dataset (Figures~\ref{fig:movielens_3} and~\ref{fig:movielens_5}), we found a clear negative correlation between the recommendation accuracy and diversity.
This reflects the well-known trade-off that focusing recommendations on popular items increases the accuracy but decreases the diversity.
Nevertheless, our DRO method consistently maintained higher diversity than the other methods at the same accuracy level.
Notably, the difference in the diversity widened as the accuracy decreased.
This result suggests that our DRO method effectively recommended unpopular items with lower recommendation frequencies.

For the Yahoo! R3 dataset (Figures~\ref{fig:r3_3} and~\ref{fig:r3_5}), we observed an atypical positive correlation; that is, recommendation accuracy increased with increasing recommendation diversity.
This implies that for this dataset, including a variety of items in the recommendations can more accurately capture users' potential preferences.
Our DRO method outperformed the other methods in terms of the diversity at the same accuracy level.
This result indicates the superiority of our DRO method, especially when recommendation diversity is prioritized.

For the BookCrossing dataset (Figures~\ref{fig:book_3} and \ref{fig:book_5}), we found a negative correlation between the recommendation accuracy and diversity, similarly to the MovieLens dataset.
In addition, for this dataset, our DRO method achieved high-quality recommendations in terms of the accuracy and diversity.
Notably, even in cases of high diversity, our DRO method maintained higher accuracy than the other methods.
This result suggests that our DRO method can simultaneously achieve high levels of accuracy and diversity in recommendations.

\begin{table*}[!t]
\centering
\caption{Computation times [s] per user for the number
of recommended items $N \in \{3, 5\}$}
\label{tab:result_data_n3}

\begin{subtable}{\textwidth}
\centering
\caption{$N=3$}
\label{tab:N=3}
\scalebox{0.95}{
\begin{tabular}{ccccccc}
\midrule
\multirow{2}{*}[-0.5ex]{Method}
    & \multicolumn{3}{c}{Parameter}
    & \multicolumn{3}{c}{Dataset} \\
\cmidrule(lr){2-4} \cmidrule(lr){5-7} 
    & $\alpha$ & $\kappa_1$ & $\kappa_2$ 
    & MovieLens & Yahoo! R3 & BookCrossing \\
\midrule
\multirow{3}{*}{MV} 
& 0.1 & \multirow{3}{*}{---} & \multirow{3}{*}{---} & 0.01 ($\pm$0.01) & 0.01 ($\pm$0.00) & 0.01 ($\pm$0.01) \\ 
& 0.3 & & & 0.01 ($\pm$0.01) & 0.01 ($\pm$0.00) & 0.01 ($\pm$0.01) \\ 
& 0.5 & & & 0.01 ($\pm$0.01) & 0.01 ($\pm$0.01) & 0.01 ($\pm$0.01) \\ 
\midrule
\multirow{3}{*}{RMV($2, 25$)} 
& 0.1 & \multirow{3}{*}{---} & \multirow{3}{*}{---} & 0.19 ($\pm$0.37) & 0.01 ($\pm$0.02) & 0.11 ($\pm$0.40) \\ 
& 0.3 & & & 0.15 ($\pm$0.29) & 0.01 ($\pm$0.02) & 0.10 ($\pm$0.37) \\ 
& 0.5 & & & 0.15 ($\pm$0.27) & 0.01 ($\pm$0.02) & 0.10 ($\pm$0.36) \\ 
\midrule
\multirow{3}{*}{RMV($4, 50$)} 
& 0.1 & \multirow{3}{*}{---} & \multirow{3}{*}{---} & 0.19 ($\pm$0.37) & 0.01 ($\pm$0.02) & 0.11 ($\pm$0.37) \\ 
& 0.3 & & & 0.15 ($\pm$0.28) & 0.01 ($\pm$0.02) & 0.10 ($\pm$0.35) \\ 
& 0.5 & & & 0.14 ($\pm$0.27) & 0.01 ($\pm$0.02) & 0.10 ($\pm$0.35) \\ 
\midrule
\multirow{4}{*}{DRO} 
& \multirow{4}{*}{---} & 0.1 & 0.1 & 3.31 ($\pm$3.27) & 0.51 ($\pm$0.74) & 2.05 ($\pm$3.45) \\ 
& & 0.1 & 5.0 & 3.31 ($\pm$3.26) & 0.52 ($\pm$0.75) & 2.45 ($\pm$6.30) \\ 
& & 5.0 & 0.1 & 3.82 ($\pm$3.87) & 0.57 ($\pm$0.85) & 2.10 ($\pm$3.50) \\ 
& & 5.0 & 5.0 & 3.20 ($\pm$3.16) & 0.58 ($\pm$1.72) & 1.63 ($\pm$3.08) \\ 
\midrule
\end{tabular}
}
\end{subtable}

\bigskip

\begin{subtable}{\textwidth}
\centering
\caption{$N=5$}
\label{tab:N=5}
\scalebox{0.95}{
\begin{tabular}{ccccccc}
\toprule
\multirow{2}{*}[-0.5ex]{Method}
    & \multicolumn{3}{c}{Parameter}
    & \multicolumn{3}{c}{Dataset} \\
\cmidrule(lr){2-4} \cmidrule(lr){5-7} 
    & $\alpha$ & $\kappa_1$ & $\kappa_2$ 
    & MovieLens & Yahoo! R3 & BookCrossing \\
\midrule
\multirow{3}{*}{MV} 
& 0.1 & \multirow{3}{*}{---} & \multirow{3}{*}{---} & 0.01 ($\pm$0.01) & 0.01 ($\pm$0.00) & 0.01 ($\pm$0.01) \\ 
& 0.3 & & & 0.01 ($\pm$0.01) & 0.01 ($\pm$0.01) & 0.01 ($\pm$0.01) \\ 
& 0.5 & & & 0.02 ($\pm$0.02) & 0.01 ($\pm$0.01) & 0.01 ($\pm$0.01) \\ 
\midrule
\multirow{3}{*}{RMV($2, 25$)} 
& 0.1 & \multirow{3}{*}{---} & \multirow{3}{*}{---} & 0.22 ($\pm$0.50) & 0.02 ($\pm$0.02) & 0.13 ($\pm$0.47) \\ 
& 0.3 & & & 0.23 ($\pm$0.49) & 0.02 ($\pm$0.04) & 0.12 ($\pm$0.40) \\ 
& 0.5 & & & 0.33 ($\pm$0.74) & 0.03 ($\pm$0.05) & 0.11 ($\pm$0.39) \\ 
\midrule
\multirow{3}{*}{RMV($4, 50$)} 
& 0.1 & \multirow{3}{*}{---} & \multirow{3}{*}{---} & 0.22 ($\pm$0.42) & 0.02 ($\pm$0.02) & 0.14 ($\pm$0.50) \\ 
& 0.3 & & & 0.22 ($\pm$0.47) & 0.02 ($\pm$0.03) & 0.11 ($\pm$0.39) \\ 
& 0.5 & & & 0.31 ($\pm$0.73) & 0.02 ($\pm$0.04) & 0.11 ($\pm$0.40) \\ 
\midrule
\multirow{4}{*}{DRO} 
& \multirow{4}{*}{---} & 0.1 & 0.1 & 3.30 ($\pm$3.28) & 0.50 ($\pm$0.72) & 2.03 ($\pm$3.43) \\ 
& & 0.1 & 5.0 & 3.30 ($\pm$3.26) & 0.50 ($\pm$0.73) & 2.21 ($\pm$4.08) \\ 
& & 5.0 & 0.1 & 3.76 ($\pm$3.79) & 0.55 ($\pm$0.81) & 1.98 ($\pm$2.69) \\ 
& & 5.0 & 5.0 & 2.82 ($\pm$2.15) & 0.58 ($\pm$1.71) & 1.63 ($\pm$3.11) \\ 
\midrule
\end{tabular}
}
\end{subtable}
\end{table*}

\subsection{Results of Computational Efficiency}\label{subsec:result_comp}

Table~\ref{tab:result_data_n3} lists the average computation time per user on the three datasets with the number of recommended items $N \in \{3, 5\}$. 
Here, the standard error of the computation time is shown in parentheses.
The risk aversion parameter was selected from $\alpha \in \{0.1, 0.3, 0.5\}$ for the MV and RMV methods, and the ambiguity parameters in Eq.~\eqref{eq:set_unc} were selected from $(\kappa_1, \kappa_2) \in \{(0.1, 0.1), (0.1, 5.0), (5.0, 0.1), (5.0, 5.0)\}$ for our DRO method. 

Overall, the computation time was longer for our DRO method than for the other methods.
This is due to the nature of Algorithm~\ref{alg:PADM}, which requires repeatedly solving SDO subproblems.
However, even under the most computationally intensive conditions, the computation time for our DRO method remained within 4 s. 
Given the requirements for offline computation or near-real-time batch processing in online recommender systems, this level of computational cost can be considered practical.

The risk aversion parameter $\alpha$ had a small and inconsistent effect on the computation time for the MV and RMV methods across different datasets. 
The same was true for the impact of the ambiguity set parameters $(\kappa_1, \kappa_2)$ on the computational efficiency of our DRO method.
On the other hand, increasing the number of recommended items from $N=3$ to $N=5$ tended to slightly increase the computation time of the RMV method and slightly decrease that of our DRO method.

\section{Conclusion}\label{sec:concl}

We proposed a computational DRO framework for selecting effective recommendations. 
We formulated a DRO model with the cardinality constraint to select a specified number of recommendations for each user.
For this optimization model, we developed the PADM algorithm to efficiently find high-quality solutions. 
We also proved the convergence properties of our PADM algorithm based on the trace-constrained and Frobenius-norm-regularized formulations. 

We evaluated the effectiveness of our recommendation selection framework through the computational experiments using the three publicly available datasets. 
Our optimization model succeeded in generating more diverse recommendations while maintaining the same recommendation accuracy as the existing recommendation selection models.
Furthermore, our solution method was able to give the aforementioned high-quality recommendations for each user in just a few seconds. 

This study opens up new possibilities for applying DRO techniques to recommender systems. 
Importantly, our method has the potential to improve the recommendation quality of various neighborhood-based and model-based methods for collaborative filtering. 

A future direction of study will be to select recommendations simultaneously for all users, under the constraint of overall recommendation quality, rather than selecting recommendations separately for each user.
Achieving this goal would involve dealing with a huge-scale optimization problem, so we need to devise algorithms that are more scalable.
Another direction of future research will be to develop more practical methods that consider fairness among items~\cite{PiSt22,UeIk24} and each user's privacy~\cite{HiSo22,YaIk25} in addition to recommendation diversity.



\backmatter



\bibliography{Ref}

\end{document}